\documentclass{amsart}
\usepackage{amsfonts,amscd, amssymb}

\newtheorem{theorem}{Theorem}[section]
\newtheorem{lemma}[theorem]{Lemma}
\newtheorem{corollary}[theorem]{Corollary}
\newtheorem{proposition}[theorem]{Proposition}
\theoremstyle{remark}

\theoremstyle{definition}

\numberwithin{equation}{section}
\makeatother

\DeclareMathOperator{\Cdb}{{\mathbb C}}
\DeclareMathOperator{\Rdb}{{\mathbb R}}

\DeclareMathOperator{\re}{{Re}}
\DeclareMathOperator{\ball}{{Ball}}

\newcommand{\oa}[1]{\operatorname{oa}(#1)}
\newcommand{\joa}[1]{\operatorname{joa}(#1)}

\begin{document}

\title{Jordan operator algebras revisited}
\author{David P. Blecher}
\address{Department of Mathematics, University of Houston, Houston, TX
77204-3008, USA}
\email[David P. Blecher]{dblecher@math.uh.edu}

\author{Zhenhua Wang}
\address{Department of Mathematics, University of Houston, Houston, TX
77204-3008, USA}
\email[Zhenhua Wang]{zhenwang@math.uh.edu}

\date{\today} 
	
\thanks{Supported by  the Simons Foundation Collaboration 
grant 527078}

\maketitle

\thispagestyle{empty} 
\setcounter{page}{1}


\begin{abstract} Jordan operator algebras are norm-closed spaces of operators
on a Hilbert space with $a^2 \in A$ for all $a \in A$.  In two recent papers by the authors and Neal, 
a theory for these spaces was developed.    It was shown  there that much of  the theory of associative operator algebras,
in particularly surprisingly much of the associative theory from several recent papers of the first author and coauthors,
 generalizes to Jordan  operator algebras.  In the present paper  we complete this task, giving  several  
results which generalize the associative case in these papers,   relating to unitizations, real positivity, hereditary subalgebras,
and a couple of other topics.   We also  solve one of the three open problems stated at the end of our earlier joint paper
on Jordan operator algebras.  
\end{abstract}

\section{Introduction}   A Jordan operator algebra is a  norm-closed linear 
space $A$ of operators
on a Hilbert space, with $a^2 \in A$ for all $a \in A$.  Or equivalently, which is
closed under the `Jordan product' $a \circ b = \frac{1}{2}(ab+ba)$.   In two recent papers \cite{BWj,BNj} 
a theory for this class  was developed.    It was shown  there that much of  the theory of associative operator algebras,
in particularly surprisingly much of the recent associative theory from e.g.\ 
\cite{BHN, BRI, BRII, BRord, BNII}, generalizes to Jordan  operator algebras.  In the present paper  we complete this task, giving  several 
results which generalize the associative case in these cited papers,  relating to unitizations, real positivity, structure of state and quasistate spaces, hereditary subalgebras, 
and a couple of other topics.   We also solve one of the three problems posed at the end of \cite{BWj}.   In that paper
 we observed that a consequence of 
Meyer's unitization theorem (see e.g.\ 2.1.15
in \cite{BLM}) is that the unitization $A^1$ of a general Jordan operator algebra $A$ 
is unique up to isometric isomorphism.   We asked if the unitization of a Jordan operator algebra is unique up to completely isometric isomorphism?    Here we answer this question in the negative. 
This outcome is not too serious for the topics contained in \cite{BWj,BNj} since most results
there only use the norm, cone of `real positive elements', and Jordan algebra structure.   

We now quickly summarize notation (for more details see the introductions to \cite{BWj,BNj}). 
If $A$ is a Jordan operator algebra acting on a Hilbert space $H$,  let $C^*(A)$ be
the $C^*$-algebra generated by $A$ in $B(H)$.   If $x,y \in A$ then $xyx \in A$ (by e.g.\ (1.3) in \cite{BWj}).  
Norm-closed Jordan subalgebras (that is, 
subspaces closed under  the Jordan product) of commutative (associative)
 operator algebras are of course ordinary (commutative associative)
operator algebras on a Hilbert space, and the Jordan product is the
ordinary product.    In particular if $a$ is an element in a
Jordan operator algebra $A$ inside a $C^*$-algebra  $B$,
 then the closed Jordan algebra
generated by $a$ in $A$ equals the closed operator algebra
generated by $a$ in $B$.    We write this as oa$(a)$ or joa$(a)$.   An element $q$ in a Jordan operator algebra $A$
is called a {\em projection} if $q^2 = q$ and $\| q \| = 1$
(so these are just the orthogonal projections on the
Hilbert space $A$ acts on, which are in $A$).     A {\em Jordan homomorphism} between  Jordan operator algebras
 is a linear map  satisfying $T(ab+ba) = T(a) T(b) + T(b) T(a)$ for $a, b \in A$, or equivalently,
that $T(a^2) = T(a)^2$ for all $a \in A$ (the equivalence follows by applying $T$ to $(a+b)^2$).  
 A {\em Jordan ideal} of  a Jordan algebra $A$
is a subspace $E$ with  $\eta \circ \xi \in E$
for $\eta \in E, \xi \in A$.  
A {\em Jordan  contractive approximate identity} (or {\em J-cai} for short)
is a net
$(e_t)$  of contractions with $e_t \circ a \to a$ for all $a \in A$.
It is shown in \cite[Lemma 2.6]{BWj} that if $A$ has a Jordan cai then it has a
{\em partial cai}, that is a net that acts as a cai for the ordinary product in
any containing generated $C^*$-algebra.       If a Jordan (or equivalently, a partial) cai for $A$ exists then $A$ is called {\em approximately unital}.  
 A {\em state} of an approximately unital Jordan 
operator algebra $A$ is a functional with $\Vert \varphi \Vert = \lim_t \, \varphi(e_t) = 1$
for some (or every) J-cai $(e_t)$ for $A$.  These extend to states of the unitization $A^1$.  They also
extend to a state (in the $C^*$-algebraic sense) on any $C^*$-algebra $B$ generated by $A$, and conversely
any state on $B$ restricts to a state of $A$.   Part of this follows easily from the fact mentioned earlier that  any partial cai for $A$ is a cai for $B$. See \cite[Section 2.7]{BWj}
for more  details.

\section{Unitization} 

Interesting examples of Jordan operator algebras can come from spaces of anticommuting operators, such as from a spin system in fermionic
analysis.  This is  related to the  
Cartan factors of type IV (certain selfadjoint spaces of operators whose squares lie in $\Cdb I$),
and the operator space $\Phi_n$ in e.g.\  \cite[Section 9.3]{Pisbk}.    
A linear space of anticommuting operators will of course (by definition of `anticommuting') have zero Jordan product.
A very simple example of this is  completely isometric to the well known operator space $R_2 \cap C_2$, which may be defined to be the
two dimensional Hilbert space with matrix norms the maximum of the `row matrix norm' and the `column matrix norm'  (the matrix norms 
of the first row or first column of the $2 \times 2$ matrix $C^*$-algebra $M_2$).   
Consider the set $E_2$ 
of 
$4 \times 4$ matrices 
$$ \left[ \begin{array}{ccccl}
0  & \alpha & \beta & 0 \\ 0 & 0  & 0  & - \beta \\
0 & 0  & 0  &  \alpha \\
0 & 0  & 0  & 0
\end{array}  \right] , \qquad \alpha, \beta  \in \Cdb. $$
This is not an associative algebra but is a Jordan operator algebra with zero Jordan product.
That is, $xy + yx = 0$ for any two such matrices, which is an anticommutation relation.     Let $G$ be the space of the same matrices but with
first column and last row removed, so that  $G \subset M_3$.
Let $F_2$ be the set of matrices in $M_2(G) \subset M_6$ which are zero  in all of  the 
first three columns and all of the last three rows.  This is an operator algebra with zero  product, which is
linearly completely isometric to $G$.  
So $E_2 \cong F_2$ completely isometrically isomorphically as  Jordan operator algebras, but not of course as associative
operator algebras. 

\begin{proposition} The two unitizations $\Cdb I_4 + E_2$ and  $\Cdb I_6 + F_2$ of
the Jordan operator algebra $E_2$ above, are not completely isometrically
isomorphic as  Jordan operator algebras, nor even as unital operator spaces.   
\end{proposition} 

\begin{proof}. The $C^*$-envelope
of $\Cdb I_4 + E_2$  is at most 16 dimensional since $\Cdb I_4 + E_2
\subset M_4$ (we shall not need this but it is easy to see that $C^*_e(\Cdb I_4 + E_2)$ is  16 dimensional).
We shall show that the dimension
of the $C^*$-envelope
of $\Cdb I_6 + F_2$  is 18.     Indeed $\Cdb I_6+ F_2$ is (via a switch of columns and rows)
completely isometrically isomorphic to the unital subalgebra $A$ of $M_3 \oplus M_3$ consisting of matrices 
$$ \left[ \begin{array}{cccl}
\lambda  & \alpha & \beta \\ 0 & \lambda  & 0 \\
0 & 0  & \lambda 
\end{array}  \right]  \; \oplus
\; \left[ \begin{array}{cccl}
\lambda  & 0 & -\beta \\ 0 & \lambda  &  \alpha \\
0 & 0  & \lambda 
\end{array}  \right]   \; \in \; M_3 \oplus M_3 , \qquad \lambda, \alpha, \beta  \in \Cdb. $$
We claim that $C^*_e(A) = M_3 \oplus M_3$.   Indeed the $*$-algebra generated by $A$ 
contains $E_{12} + E_{56}, E_{13} - E_{46}, (E_{12} + E_{56})(E_{31} - E_{64}) = - E_{54},$ hence also 
$E_{45}, E_{45}(E_{12} + E_{56}) = E_{46},$ and thus $E_{56}.$  Indeed it contains
 $E_{ij}$ for $4 \leq i, j \leq 6.$   Hence also $E_{12}, E_{1,3}$, 
 indeed $E_{ij}$ for $1 \leq i, j \leq 3.$   So $A$ generates $M_3 \oplus M_3$.   Let $J$ be a nontrivial
 ideal in  $M_3 \oplus M_3$, such that the canonical map  $A \to (M_3 \oplus M_3)/J$ is a complete
isometry.   We may assume $J = 0 \oplus M_3$, the contrary case being similar.   That is
the canonical compression of $A$ to its first $3 \times 3$ block (that is, multiplication on $A$ by $I_3 \oplus 0$)
is completely isometric.   Note that this implies that $R_2 \cap C_2$ is completely isometric to $R_2$,
which is known to be false. Thus $J = (0)$, so that  $C^*_e(A) = M_3 \oplus M_3$ as desired.  
\end{proof}

\section{Real positivity} \label{repo}
  If $A$ is a  
Jordan operator algebra we define ${\mathfrak r}_A$ to be the
elements $x$ in $A$ with positive real part (that is, $x + x^* \geq 0$).   We call these the {\em real positive} or {\em accretive} elements of $A$. 

Suppose that $A$ is an approximately unital Jordan operator algebra acting on a Hilbert space $H$, and let $C^*(A)$ be
the $C^*$-algebra generated by $A$ in $B(H)$.   As we said at the end of the introduction, states  
of $A$ are precisely the restrictions to $A$ of states on $C^{\ast}(A)$.
We write $S(A)$ for the state space of $A$.  
Let $Q(A)$ be the quasistate space of $A,$ namely $Q(A)=\{t\varphi, t\in[0,1], \varphi\in S(A)\}.$   The  quasistates of $A$ are precisely the restrictions to $A$ of quasistates on $C^{\ast}(A)$.  It follows that the quasistate space $Q(A)$ is convex and weak* compact. Indeed if $\varphi_t\in S(C^{\ast}(A)),$ $\lambda_t\in [0,1],$ $\varphi\in A^{\ast},$ and $\lambda_t\varphi_t\to \varphi$ weak* on $A,$ then there exists a $\psi\in Q(C^{\ast}(A))$ and convergent subnets $\lambda_{t_{\mu}}\to \lambda\in[0,1],$ $\varphi_{t_{\mu}}\to \psi$ weak* on $C^{\ast}(A).$ This forces $\varphi=\lambda \psi\in Q(A).$ So $Q(A)$ is weak* compact.   
Note that $S(A)$ is convex, for example by \cite[Lemma 2.21 (4)]{BNj}, and from this it is easy to see that $Q(A)$ is convex.  We set ${\mathfrak F}_A = \{ x \in A: \| 1 - x \| \leq 1 \}$.  Here $1$ is the identity of the unitization $A^1$ if $A$ is nonunital.     We also have (see the end of Section 2.7 in \cite{BWj}) 
$${\mathfrak r}_A=\{x\in A:\re\varphi(x)\geq 0,\ \mbox{for all}\ \varphi\in S(A)\}. $$
Define  $${\mathfrak c}_{A^{\ast}}=\{\varphi\in A^{\ast}:\re\varphi(x)\geq 0,\ \mbox{for all}\ x\in {\mathfrak r}_A\}. $$
Similarly ${\mathfrak c}_{A^{\ast}}^{\Rdb}$ are the real linear functionals that are positive
on ${\mathfrak r}_A$.
The above  are the {\em real positive 
 functionals.}
We will say an approximately unital Jordan operator algebra $A$ is {\em scaled} if ${\mathfrak c}_{A^{\ast}}=\Rdb^+S(A).$

The following results are the 	approximately unital Jordan versions of most of 2-6--2.8 in  \cite{BOZ}.

\begin{proposition} \label{scaled}
	Let $A$ be an approximately unital nonunital Jordan operator algebra.  Then $Q(A)$ is  the weak* closure of $S(A).$	Also, a functional $f \in {\mathfrak c}_{A^{\ast}}$ if and only if $f$ is a nonnegative multiple of  a state.	
That is,  an approximately unital nonunital Jordan operator algebra is scaled.    
\end{proposition}

\begin{proof}  
That  $Q(A)$ is  the weak* closure of $S(A)$ will follow because this is true for $C^*$-algebras
and because of the fact, stated earlier, that states of $A$ are precisely the restrictions to $A$ of states on $C^{\ast}(A)$. 
  The last assertion follows from the fact above that $Q(A)$ is weak* closed and the argument for \cite[Lemma 2.7 (1)]{BOZ}.   Alternatively, if  $f$ is real positive, then $\re f$ is real positive (i.e.\ $\re f({\mathfrak r}_A) \subset [0,\infty)$),   which implies by  \cite[Lemma 4.13]{BWj} that $\re f =\lambda\re \varphi$ for some $\lambda\geq 0$ and $\varphi\in S(A).$ Therefore, $f=\lambda \varphi$ by the uniqueness of the  extension of $\re f.$
\end{proof}

\begin{corollary}
	If $A$ is an approximately unital nonunital Jordan operator algebra, then:
	\begin{itemize}
		\item[(i)] $S(A^1)$ is the convex hull of trivial character $\chi_0$ on $A^1$ (which annihilates $A$) 
and the set of states on $A^1$ extending states of $A.$
		\item[(ii)] $Q(A)=\{\varphi_{\vert_A}: \varphi\in S(A^1)\}.$ 
	\end{itemize}
\end{corollary}
\begin{proof}   These follow easily from the fact that they are  true for $C^*$-algebras
and states of $A$ are precisely the restrictions to $A$ of states on $C^{\ast}(A)$.   Or one may deduce them e.g.\ from the fact that $A$ is scaled as in the proof of \cite[Lemma 2.7]{BOZ}.  
\end{proof}

\begin{lemma} {\rm (Cf.\ \cite[Lemma 6.6]{BOZ}.)} \ 
	Suppose $A$ is an approximately unital Jordan operator algebra.
	\begin{itemize}
		\item[(1)] The cones ${\mathfrak c}_{A^{\ast}}$ and ${\mathfrak c}_{A^{\ast}}^{\Rdb}$  are additive (that is, the norm on $A^{\ast}$ is additive on these cones).
		\item[(2)] If $(\varphi_t)$ is an increasing net in ${\mathfrak c}_{A^{\ast}}^{\Rdb}$	which is bounded in norm, then the set converges in norm, and its limits is the least upper bound of the net.
	\end{itemize}	
\end{lemma}
\begin{proof}
	This is as in \cite[Lemma 6.6]{BOZ}, however one needs to appeal to \cite[ Lemma 4.13]{BWj} in place of the matching result used there.  
\end{proof}

\begin{corollary} {\rm (Cf.\ \cite[Corollary 6.9]{BOZ}.)} \ 
	Let $A$ be an approximately unital Jordan operator algebra. If $f\leq g\leq h$ in $B(A,\Rdb)$ in the natural `dual ordering' induced
by $\preceq$, then $\Vert g\Vert\leq \Vert f\Vert+\Vert h\Vert.$
\end{corollary}
\begin{proof}
	For any $x\in A$ with $\Vert x\Vert<1,$ then by \cite[Theorem 4.1(4)]{BWj} there exists $a, b\in \frac{1}{2}{\mathfrak F}_A\subset\ball (A)$ such that $x=a-b.$ If $g(x)\geq 0,$ then $$g(x)=g(a)-g(b)\leq h(a)-f(b)\leq \Vert h\Vert+\Vert f\Vert.$$ If $g(x)\leq 0,$ then $$\vert g(x)\vert=g(b)-g(a)\leq h(b)-f(a)\leq \Vert h\Vert+\Vert f\Vert.$$ 
	Therefore, $\Vert g\Vert\leq \Vert f\Vert+\Vert h\Vert.$ 	
\end{proof}

 In \cite[Proposition 3.27]{BWj}, we show that if $J$ is an approximately unital closed two-sided Jordan ideal in a
Jordan operator algebra $A,$
then $A/J$ is (completely isometrically isomorphic to)
a Jordan operator algebra.   
The following is the approximately unital Jordan version of  \cite[Corollary 8.9]{BOZ}.    Below $W_A(x) = \{ \varphi(x) : \varphi \in S(A) \}$ is the numerical range of $x$ in $A$.  
 
\begin{proposition} \label{tri}
	Suppose that $J$ is an approximately unital Jordan ideal in a unital Jordan operator algebra $A.$ Let $x\in A/J$ with $K=W_{A/J}(x).$ Then 
	\begin{itemize}
		\item[(i)] If $K$ is not a nontrivial line segment in the plane, then there exists $a\in A$ with $a +J = x, \Vert a\Vert=\Vert x\Vert$, and  $W_A(a)=W_{A/J}(x).$	
		\item[(ii)] If $K=W_{A/J}(x)$ is a nontrivial line segment, let $\hat{K}$ be any thin triangle with $K$ as one of the sides. Then there exists $a\in A$ with $a + J = x, \Vert a\Vert=\Vert x\Vert$, and  $K\subset W_A(a)\subset \hat{K}.$
	\end{itemize}
\end{proposition}
\begin{proof}       For (i) note that if the numerical range of $x$ is not a singleton nor a line segment then this follows from \cite[Theorem 8.8]{BOZ} since $J$ is an $M$-ideal \cite{BWj}.  
If the numerical range of an element  $w$ 
in a 
unital Jordan algebra  is a singleton then the same is true with respect to the generated $C^*$-algebra, so that $w \in \Cdb 1$.

 If $W_{A/J}(x)$ is a nontrivial line segment $K$,  this works just as in the proof of  \cite[Corollary 8.9]{BOZ}.  Replace $A$ by the unital Jordan algebra $B=A\oplus^{\infty}\Cdb,$ replace $J$ by the approximately unital Jordan ideal $I=J\oplus (0).$    Then $I$ is an $M$-ideal in $B$ by \cite[Theorem 3.25 (1)]{BWj}.
For a scalar $\lambda$ chosen as in the proof of  \cite[Corollary 8.9]{BOZ}, $W((x,\lambda))$ is the convex hull
of $K$ and $\lambda$, hence has nonempty interior.   Since  $I$ is an $M$-ideal in $B$, 
 we may   appeal to  \cite[Theorem 8.8]{BOZ}
in the same way as in the proof we are copying, to obtain our result.   
\end{proof}

\begin{corollary}  \label{qra}  {\rm (Cf.\ \cite[Corollary 8.10]{BOZ}.)} \ 
	If $J$ is an approximately unital Jordan ideal in any (not necessarily approximately unital) Jordan operator algebra $A,$  then $q({\mathfrak r}_A)={\mathfrak r}_{A/J}$ where $q: A \to A/J$ is the quotient map.		
\end{corollary}

\begin{proof}    We know from  \cite[Proposition 3.28]{BWj} that $q({\mathfrak F}_A)={\mathfrak F}_{A/J}$, and 
${\mathfrak r}_A = \overline{\Rdb_+ \, {\mathfrak F}_A}$ by the proof of \cite[Theorem 3.3]{BRII}.   
So $q( \Rdb_+ \, {\mathfrak F}_A)=  \Rdb_+ \, {\mathfrak F}_{A/J}$.   Taking closures we have
 $q({\mathfrak r}_A) \subset {\mathfrak r}_{A/J}$.
	
The other direction uses Proposition \ref{tri}.
	If $A$ is unital the  result immediately follows from Proposition \ref{tri} 
as in  \cite[Corollary 8.10]{BOZ}.  If $A$ is nonunital and $A/J$ is nonunital, then by \cite[Proposition 2.4 and Proposition 3.27]{BWj}, we can extend $q$ to a contractive unital Jordan homomorphism from $A^1$ to a unitization of $A/J.$ Then $J$ is still an M-ideal in $A^1$ by \cite[Theorem 3.25 (1)]{BWj}. Therefore, again the result follows
as in  \cite[Corollary 8.10]{BOZ}  by applying  the unital case to the canonical map from $A^1$ onto $(A/J)^1 = A^1/J.$  (The latter formula
following from  the Jordan operator algebra case of Meyer's unitization theorem  \cite[Section 2.2]{BWj}.)   
		If $A/J$ is unital, then one may reduce to the previous case where it is not unital, by the trick for that in the proof of
  \cite[Corollary 8.10]{BOZ}.  
\end{proof}

 We recall that an element $x$ in a unital Jordan 
algebra $A$ is invertible if there exists $y \in A$ with $x \circ y = 1$ and $x^2 \circ y = x$.
If $A$ is an associative algebra and $x \circ y = \frac{1}{2}(xy + yx)$ then 
it is known (or is an easy exercise to check) that this coincides with the usual definition.
For a  Jordan operator algebra the spectrum Sp$_A(x)$ is defined to be the scalars $\lambda$ such that 
$\lambda 1 - x$ is invertible in $A^1$, and  as in the Banach algebra case
can be shown to be a  
compact nonempty set, on whose complement
$(\lambda 1 - x)^{-1}$ is analytic as usual, and the spectral radius is the 
usual $\lim_n \, \| x^n \|^{\frac{1}{n}}$.   These facts are all well known in the theory of Jordan Banach algebras.  For a general (possibly nonunital) 
Jordan operator algebra  we say that $x$ is {\em quasi-invertible}  if $1-x$ is invertible in $A^1$.

\begin{theorem}
\label{cxAx} {\rm (Cf.\ \cite[Theorem 3.2]{BRI} and  \cite[Theorem 3.21]{BOZ}.)} \ For a Jordan operator algebra $A$, if $x\in {\mathfrak r}_A$, the following are equivalent:
\begin{itemize}
\item[(i)] $x \, \joa{x} \,  x$ is closed.
\item[(ii)]	$\joa{x}$ is unital.
\item[(iii)]	 There exists $y\in\joa{x}$ with $xyx=x.$
\item[(iv)] $xAx$ is closed.
\item[(v)]	 There exists $y\in A$ such that $xyx=x.$
Also, the latter conditions imply
\item[(vi)] $0$ is isolated in or absent from ${\rm Sp}_A (x).$
\end{itemize}
Finally, if further $\joa{x}$ is semisimple, then condition {\rm (i)-(vi)} are all equivalent.
\end{theorem}
\begin{proof}
If $A$ is unital, $x\in {\mathfrak r}_A,$ and $x$ is invertible in $A$, then $s(x) = 1 \in \joa{x}$.
Thus (ii) is true, and in this case (i)-(vi) are all obvious. So we may assume that $x$ is not invertible in $A.$   The equivalence of (i)--(iii) is the same as the operator algebra case (since $\joa{x} = {\rm oa}(x)$); and these imply (v).

	(iv) $\Rightarrow$ (v) \ Suppose that $xAx$ is closed. Now $x=(x^{\frac{1}{3}})^3,$ and since $x = \lim_n \, x^{\frac{1}{n}} \, 
x  x^{\frac{1}{n}}$ and $x^{\frac{1}{n}} \in {\rm oa}(x)$ we see that    
	$$x^{\frac{1}{3}}\in  {\rm oa}(x) =\overline{x  {\rm oa}(x) x}\subset \overline{xAx}=xAx,$$
	and so $xyx=x$ for some $y\in A.$ 
	
	(v) $\Rightarrow$ (iv)  \ Similarly to the original proof of  \cite[Theorem 3.2]{BRI}, $$xAx=(xyx)A(xyx)\subseteq xyAyx\subseteq xAx,$$ 
so that $xAx=xyAyx$. Since $xy$  is an idempotent (in any containg $C^*$-algebra), $xAx$ is closed.

Clearly (v) implies the same assertion but with $A$ replaced by $C^*(A)$.  This, by the  theorem we are generalizing, 
implies (i)--(iii).

That 
(v) implies  (vi) is as in   \cite[Theorem 3.2]{BRI}: If $0$ is not isolated in ${\rm Sp}_{A}(x)$, then there is a sequence of boundary points in ${\rm Sp}_A(x)$ converging to $0$. Since $\partial {\rm Sp}_A(x)\subset {\rm Sp}_{B(H)}(x)$ as in the Banach algebra case, we obtain a contradiction.
Similarly 	(vi) implies (ii) if oa$(x)$ is semisimple:
if the latter is not unital, then $0$ is an isolated point in ${\rm Sp}_{A}(x)$ iff  $0$ is isolated in ${\rm Sp}_{{\rm oa}(x)}(x),$ and so we can use the original proof. 
\end{proof}

	\section{Hereditary subalgebras} 
	
 If  $A$   is a Jordan subalgebra of a $C^*$-algebra $B$ then $A^{**}$ may be identified with the
Jordan subalgebra $A^{\perp \perp}$ of the $W^*$-algebra $B^{**}$, and hence  $A^{**}$ is a  Jordan operator algebra
with its `Arens product' (see e.g.\ \cite[Section 1.3]{BWj}).  
We recall that a {\em hereditary subalgebra} (or {\em HSA}) of a Jordan operator algebra $A$ is a norm-closed Jordan subalgebra $D$ of $A$ containing a
Jordan contractive approximate identity such that $xAx \subset D$ for all $x \in D$. In \cite{BWj} we showed  that a subspace
$D$ of $A$ is a hereditary subalgebra if and only if $D$ is of form
$\{ a \in A : a = pap \}$ for a projection
$p \in A^{**}$ which is what we called $A$-{\em open} in \cite{BWj} (that is, there exists a net in $A$ with
$x_t = p x_t p \to p$ weak* in $A^{**}$).  In \cite{BNj} it is shown that 
if in addition  $A$   is a Jordan subalgebra of a $C^*$-algebra $B$ then $p$ is $A$-open if and only if $p$ is  open in Akemann's sense (see e.g.\ \cite{Ake2}) as a projection in $B^{**}$.  If $A$ is an approximately unital Jordan operator algebra then a projection $q$ in $A^{**}$ is {\em closed} if
$q^\perp = 1-q$ is open in $A^{**}$, where $1$ is the identity of $A^{**}$.

By a trivial HSA of $A$ below we mean of course $(0)$ or $A.$    We refer the reader to \cite{BWj,BNj} for the theory of 
hereditary subalgebras of Jordan operator algebras.   In the present section we generalize to the Jordan case the main remaining aspects 
from the associative theory of HSA's, and their relation to open and closed projections, from papers such as \cite{BHN,BRI, BRord}.

\begin{theorem} {\rm (Cf.\ \cite[Theorem 4.1]{BRI}.)} \ 
For a unital Jordan operator algebra $A$, the following are equivalent:
\begin{itemize}

\item[(i)] $A$ has no nontrivial HSA's (equivalently, $A^{**}$ has no nontrivial open projections).  
\item[(ii)] $a^n\rightarrow  0$ for all $a\in {\rm Ball}(A)\setminus \Cdb 1.$
\item[(iii)] The spectral radius $r(a)< \Vert a\Vert$ for all $a\in {\rm Ball}(A)\setminus \Cdb 1.$
\item[(iv)] The numerical radius $v(a)< \Vert a\Vert$ for all $a\in {\rm Ball}(A)\setminus \Cdb 1.$
\item[(v)] $\Vert 1+a\Vert <2$ for all $a\in {\rm Ball}(A)\setminus \Cdb 1.$
\item[(vi)]	${\rm Ball}(A)\setminus \Cdb 1$ consists entirely of   elements $x$ which are quasi-invertible in $A$. 
\end{itemize}
	 If $A$ has a partial cai then the following are equivalent:
\begin{itemize}

\item[(a)] $A$ has no nontrivial HSA's.

\item[(b)] $A^1$ has one nontrivial HSA.

\item[(c)] ${\rm Re}(x)$ is strictly positive for every $x \in {\mathfrak F}_A\setminus\{0\} .$   
\end{itemize} 
\end{theorem}
  
\begin{proof}    Clearly (ii)--(v) are equivalent by the theorem we are generalizing applied to oa$(a)$.   

(i)  $\Rightarrow$ (vi) \ The  HSA's are in
bijective correspondence with the  open projections in $A^{**}$, and these are the same thing as suprema of support projections 
of elements in ${\mathfrak F}_A$ (or equivalently of real positive  elements in $A$ by \cite[Lemma 3.12 (1)]{BWj}).  
Thus, $A$ has no nontrivial HSA's iff  $s(x)$ is an identity for $A^{**}$ for all
$x \in {\mathfrak F}_A\setminus\{0\} .$      Let $B = {\rm oa}(1,a)$ for  $a\in {\rm Ball}(A) \setminus \{ 1
\}.$    Then $B \subset \overline{(1-a) A (1-a)} = A$, so that 
$(1-a)^{\frac{1}{n}}$  is an approximate identity for $B$, which must therefore converge to $1$.    It follows that $B =  {\rm oa}(1-a)$,
and so by the Neumann lemma (approximating $1$) we have $1-a$ is invertible in $B$, hence in $A$, so that (vi) holds. 

(vi) \ $\Rightarrow$ (i) \ Conversely $a \in {\rm Ball}(A) \setminus \{ 1
\}$  quasi-invertible, with quasi-inverse $a'$,   implies that $$(1-a) (1-a') A (1-a')  (1-a) = A \subset 
(1-a)  A  (1-a) \subset A,$$ so that $s(1-a) = 1$.  Hence  (vi) implies that every nonzero element in ${\mathfrak F}_A$ has support
projection $1$.  

(iii) $\Rightarrow$ (vi) \ If $r(a)< \Vert a\Vert$ for all $a\in {\rm Ball}(A)\setminus \Cdb 1,$ then $a$ is quasi-invertible, and its quasi-inverse
is in oa$(1,a)$.

 (i) $\Rightarrow$ (v) If there exists $a\in {\rm Ball}(A)\setminus \Cdb 1$ such that $\Vert 1+a\Vert =2$, then $\Vert a \Vert =1$. 	By the Hahn-Banach Theorem, $\exists$ $\varphi\in {\rm Ball}(A^*)$ such that $\varphi(1+a)=2$, which means that $\varphi(1)=\varphi(a)=1.$ Hence
$\varphi$ is a state and  $\varphi(1-a)=0.$ This implies $s(1-a) \neq 1$ by Lemma 3.10 in \cite{BWj}, which contradicts (i) by the relationship between HSA's and $s(x)$ mentioned at the start of the proof.  

That (b) implies (a) is obvious.    For the converse note that if $e= 1_{A^{**}}$ is the (central) support projection of $A$ then for an open projection $p$ in $(A^1)^{**}$,  $ep$  is open by 
\cite[Proposition 6.4]{BNj}.     (We may write the `non-Jordan expression' $ep$  since $e$ is central; thus if one likes such expessions
below may be evaluated in any containing generated $C^*$-algebra.)
Note that $1-e$ is a minimal projection in $(A^1)^{**}$ since
$(1-e) (a + \lambda 1) = \lambda (1-e)$ for all $a \in A, \lambda \in \Cdb$.
So if (a) holds then $ep = e$ or $ep = 0$, whence $p = e$ or $p = 1$, or $p = 0$ or $p = 1-e$.   The last of these is impossible,
since if $x_t = (1-e) x_t \to 1-e$  with $x_t \in A^1$, then $e x_t = 0$, 
which implies that $x_t$  is in the kernel of the character on $A^1$ that annihilates $A$.
So $x_t \in \Cdb 1$, giving the contradiction $1-e \in \Cdb 1$.   
That (a) is equivalent to (c) directly follows from \cite[Lemma 3.11]{BWj} (which gives  ${\rm Re}(x)$ is strictly positive  iff
$s(x) = 1$) and the relationship between HSA's and $s(x)$ mentioned at the start of the proof.   This equivalence holds for unital $A$ too.  
\end{proof}

{\bf Remark.}  The last result has similar corollaries as in the associative algebra case.   For example one may deduce 
that an approximately unital Jordan operator algebra with no countable Jordan cai, has nontrivial HSA's. 
To see this note that if $A$ has no countable partial cai then by \cite[Theorem 3.22]{BWj} there is no element $x\in {\mathfrak r}_A$ with $s(x)=1_{A^{**}}.$ Thus  by the previous proof there are  
nontrivial HSA's in $A$.

The following are  Jordan variants of Theorem 7.1 and  Corollary 7.2 in \cite{BRI}.  
	
\begin{theorem}
If $A$ is an approximately unital Jordan operator algebra which is a closed two-sided Jordan ideal in an operator algebra $B,$ then $\overline{xAx}$ is a HSA in A for all  $x\in {\mathfrak r}_B.$ 	
\end{theorem}
\begin{proof}
For any $x\in {\mathfrak r}_B,$ then by the operator algebra case (see e.g.\ \cite[Theorem 3.1]{BRII}) we know that  $\oa{x}$ has a countable 
cai  $(e_n)$ say.   Let $J=\overline{xAx}$, then $yAy \subset J$ for all $y \in J$.  Suppose that $(f_t)$ is a cai for $A,$ then $(e_nf_te_n)\in J$ is a Jordan cai for $J$ by routine techniques. So $\overline{xAx}$ is a HSA in $A$ by \cite[Proposition 3.3]{BWj}.
\end{proof}

\begin{corollary}
	If $A$ is an approximately unital Jordan operator algebra, and if 
 $\eta$ is a real positive element in the Jordan multiplier algebra of $A$ (see Definition {\rm 2.25} in {\rm  \cite{BWj}}),  then $\overline{\eta A\eta}$ is a HSA in $A.$
\end{corollary}

The following is the Jordan version of \cite[Corollary 2.25]{BRI}:

\begin{corollary}
Let $A$ be a unital  Jordan subalgebra of $C^*$-algebra $B$ and let $q\in A^{**}$ be a closed projection associated with an HSA
$D$  in $A$. Then an explicit Jordan cai for $D$ is given by $x_{(u,\epsilon)}=1-a,$ where $a$ is an element which satisfies the conclusions of the noncommutative Urysohn Theorem {\rm  9.5} in {\rm \cite{BNj}}, for  an open projection $u\geq q,$ and a scalar $\epsilon>0.$ This Jordan cai is indexed by such pairs $(u,\epsilon)$, that is, by the product of the directed set of open projections $u\geq q,$ and the set of $\epsilon>0.$ This Jordan cai is also in $\frac{1}{2}{\mathfrak F}_A$.
\end{corollary} 
\begin{proof}  Certainly $x_{(u,\epsilon)}q=(1-a)q=q-q=0,$ 
and similarly $qx_{(u,\epsilon)} = 0$, so that $x_{(u,\epsilon)}\in D.$ Also, $\Vert x_{(u,\epsilon)}\Vert\leq 1,$ indeed $\Vert 1-2x_{(u,\epsilon)}\Vert\leq 1.$ 
The proof in  \cite[Corollary 2.25]{BRI} shows that $x_{(u,\epsilon)}b \to b$ in $B$ for $b\in \ball(D).$   Similarly
$b x_{(u,\epsilon)} \to b$.   So $x_{(u,\epsilon)}b + bx_{(u,\epsilon)}\to 2b$ in $A$.  
Therefore, $(x_{(u,\epsilon)})$ is a Jordan cai for $D.$
\end{proof}

Finally we discuss 4.10--4.12 in \cite{BRord}.  As mentioned towards the end of \cite{BNj} the proof of the first of these 
results mistakenly states that $D$ is
always approximately unital, and this is used in the last lines of the proof and in 4.12.   The rest of the proof is correct and 
generalizes to the Jordan case though (indeed most of this is done explicitly in \cite[Theorem 8.7]{BNj}).     Thus if 
 $D$ is approximately unital, then Theorems 4.10 and 4.12 in \cite{BRord}, and their Jordan algebra generalizations are correct.   That is, we have:
 
We recall that an element $b$ in a Jordan operator algebra $A$  {\em commutes} with a projection $p$ if $p b = b p$ in some $C^*$-algebra
containing $A$ as a Jordan subalgebra. 
 By (1.1) in \cite{BWj} this is equivalent to $b \circ p = pbp$.

\begin{theorem} \label{peakth2}
 Suppose that $A$ is a Jordan operator algebra, $b  \in A$, and 
 $p$ is an open projection in $A^{**}$ commuting with $b$, such that 
  $\Vert p^\perp b p^\perp  \Vert \leq 1$ (here $p^\perp = 1-p$ where $1$ is the identity of the unitization of $A$ if $A$ is nonunital).
Suppose also that $\| p^\perp (1-2b) p^\perp  \| \leq 1$.  
Then  there exists $g \in \frac{1}{2} {\mathfrak F}_A
\subset {\rm Ball}(A)$ commuting with $p$ such that $p^\perp g p^\perp  =  p^\perp b p^\perp$.   Indeed such $g$ may be chosen
`nearly positive' in the sense of the introduction to  {\rm \cite{BRord}}.  
\end{theorem}

\begin{proof}   As we said above most of this proof is done explicitly in \cite[Theorem 8.7]{BNj}.
In the latter paper in the present setting the algebra $D$ in the proof of  \cite[Theorem 4.10]{BRord} is
written as $\{ x \in A \cap C : q \circ x = 0 \}$ where $q = 1-p$.
So $D$ equals, by facts mentioned
in \cite[Theorem 8.7]{BNj}, 
$$\{ x \in A \cap C : x = p x p \} = A \cap C \cap \tilde{D} = A  \cap \tilde{D} = 
\{ x \in A  : x = p x p \},$$ 
the HSA in $A$ with support $p$.  So $D$ is approximately unital.   Now  the reader familiar with \cite{BWj,BNj} will have no trouble 
checking the Jordan variant of the last few lines of the proof of  \cite[Theorem 4.10]{BRord}, 
 using \cite[Proposition 3.28]{BWj} in place
of the analogous result used there.
(We remark that even in the case that $p= 1-q \notin A^{**}$ the quotient  $I/D$ may be viewed by the 
centered equation in the proof, as
a Jordan subalgebra (even Jordan ideal) of  $C/J \subset qC^{**} \subset q (A^1)^{**} q$, hence $I/D$ is a Jordan operator algebra.)
\end{proof}

\begin{lemma} \label{mustuse} {\rm (Cf.\ \cite[Lemma 4.11]{BRord}.)} \ 
	Suppose that $A$ and $B$ are closed Jordan subalgebras of unital Jordan operator algebras $C$ and $D$ respectively, with $1_C\notin A$ and $1_D\notin B,$ and $q:A\to B$ is a $1$-quotient map 
	(that is, induces an isometry $A/\ker(q) \to B$) and Jordan homomorphism such that $\ker(q)$ is approximately unital. Then the unique unital extension of $q$ to a unital map from $A+\Cdb 1_C$ to $B+\Cdb 1_D$ is a $1$-quotient map.	
\end{lemma}
\begin{proof}
	Let $J=\ker(q)$, let $\tilde{q}:A/J\to B$ be the induced isometry, and let $\theta:A+\Cdb1_C\to B+\Cdb 1_D$ be the unique unital extension of $q.$ This gives a one-to-one Jordan homomorphism $\tilde{\theta}: (A+\Cdb1_C)/J\to B+\Cdb 1_D$ which equals $\tilde{q}$ on $A/J.$ If $B,$ and hence $A/J$, is not unital, then $\tilde{\theta}$ is an isometric Jordan isomorphism by \cite[Corollary 2.5]{BWj}. Similarly, if $B$ is unital, then $\tilde{\theta}$ is an isometric Jordan isomorphism by the uniqueness of the unitization of an already unital Jordan operator algebra. So in either case we may deduce that $\tilde{\theta}$ is an isometric Jordan isomorphism and $\theta$ is a $1$-quotient map.	
\end{proof}

The following generalizes part of Corollary \ref{qra}.

\begin{theorem} \label{peakthang2}  {\rm (A noncommutative Tietze theorem)} \
 Suppose that $A$ is a Jordan  operator algebra
(not necessarily approximately unital),
and that  $p$ is an open projection in $A^{**}$.  Set $q = 1-p \in (A^1)^{**}$.
Suppose that $b   \in A$ commutes with $p$, and $\Vert q b q \Vert \leq 1$, and that  the numerical range of 
$q bq$ (in $q (A^1)^{**}q$ or $(A^1)^{**}$) 
is contained in a compact convex set $E$ in the plane.  We also suppose, by
fattening it slightly if necessary, that $E$ is not a line segment.  Then there exists $g \in {\rm Ball}(A)$ commuting with $p$,
  with $q g q  = q b q$, 
such that  the numerical range of 
$g$ with respect to  $A^1$ is contained in $E$.
\end{theorem}

\begin{proof}   Similar remarks as in the  proof of Theorem \ref{peakth2} apply here; except in addition 
one must use
Proposition \ref{tri} and Lemma \ref{mustuse}  in place
of the analogous result used in the proof of  \cite[Theorem 4.12]{BRord}.
Also one uses an obvious Jordan operator algebra generalization of
the direct sum of Banach algebras trick used there (see also the proof of 
Proposition \ref{tri} above).  \end{proof}

\end{document}